\documentclass[10pt]{amsart}

\usepackage{amsmath}
\usepackage{amssymb}

\newtheorem{theorem}{Theorem}[section]
\newtheorem{proposition}[theorem]{Proposition}
\newtheorem{lemma}[theorem]{Lemma}

\theoremstyle{definition}

\newtheorem{remark}[theorem]{Remark}

\numberwithin{equation}{section}

\newcommand{\N}{\mathbb{N}}                        % natural numbers
\newcommand{\R}{\mathbb{R}}                        % real numbers
\newcommand{\C}{\mathbb{C}}                        % complex numbers
\newcommand{\vp}{\varphi}                          % mapping phi
\newcommand{\VV}{\mathcal{V}}                      % zero set germ
\newcommand{\FF}{\mathcal{F}}                      % sheaf module F
\newcommand{\GG}{\mathcal{G}}                      % sheaf module G
\newcommand{\II}{\mathcal{I}}

\newcommand{\OO}{\mathcal{O}}                      % structure sheaf
\newcommand{\OYeta}{{\mathcal{O}}_{Y,\eta}}        % local ring of Y at eta
\newcommand{\anTor}{\widetilde{\mathrm{Tor}}}      % analytic Tor functor
\newcommand{\antens}{\tilde{\otimes}}              % analytic tensor product
\newcommand{\tensR}{\tilde{\otimes}_R}             % analytic tensor product over R
\newcommand{\mm}{\mathfrak{m}}                     % ideal m 
\newcommand{\nn}{\mathfrak{n}}                     % ideal n 
\newcommand{\supp}{\mathrm{supp}}                  % support
\newcommand{\rank}{\mathrm{rank}\,}                % rank of a matrix
\newcommand{\im}{\mathrm{im}\,}                    % image of a map

\newcommand{\KK}{\mathbb{K}}                       % field K
\newcommand{\tx}{\tilde{x}}                        % x tilde [=(x_1,...,x_{m-1})]
\newcommand{\txi}{\tilde{\xi}}                     % xi tilde [=(xi_1,...,xi_{m-1})]
\newcommand{\vr}{\varrho}                          % letter rho, only nicer
\newcommand{\lb}{\lambda}	                   % letter lambda

\newcommand{\id}{\mathrm{id}}              % identity map

\begin{document}

\title{An inductive analytic criterion for flatness}

\author{Janusz Adamus, Edward Bierstone and Pierre D. Milman}
\address{J. Adamus, Department of Mathematics, The University of Western Ontario, London, Ontario, Canada N6A 5B7
         and Institute of Mathematics, Jagiellonian University, ul. {\L}ojasiewicza 6, 30-348 Krak{\'o}w, Poland}
\email{jadamus@uwo.ca}
\address{E. Bierstone, The Fields Institute, 222 College Street, Toronto, Ontario, Canada M5T 3J1
         and Department of Mathematics, University of Toronto, Toronto, Ontario, Canada M5S 2E4}
\email{bierston@fields.utoronto.ca}
\address{P.D. Milman, Department of Mathematics, University of Toronto, Toronto, Ontario, Canada M5S 2E4}
\email{milman@math.toronto.edu}
\thanks{Research partially supported by Natural Sciences and Engineering
Research Council of Canada Discovery Grant OGP 355418-2008 and Polish Ministry 
of Science Discovery Grant NN201 540538 (Adamus), and NSERC Discovery 
Grants OGP 0009070 (Bierstone), OGP 0008949 (Milman)}



\begin{abstract}
We present a constructive criterion for flatness
of a morphism of analytic spaces $\vp: X \to Y$ (over $\KK = \R$ or $\C$)
or, more generally, for flatness over $\OO_Y$ of a coherent sheaf of $\OO_X$-
modules $\FF$. The criterion is a combination of a simple linear-algebra
condition ``in codimension zero'' and a condition ``in codimension
one'' which can be used together with the Weierstrass preparation theorem
to inductively reduce the fibre-dimension of the morphism $\vp$.
\end{abstract}
\maketitle

%%%%%%%%%%%%%%%%%%%%%%%%%%%%%%%%%%%%%%%%%%%%%%%%%%
%%%%%%%%%%%%%%%%%%%%%%%%%%%%%%%%%%%%%%%%%%%%%%%%%%
%%%%%%%%%%%%%%%%%%%%%%%%%%%%%%%%%%%%%%%%%%%%%%%%%%
\section{Introduction}
\label{sec:intro}

The main result of this article is a constructive criterion for flatness
of a morphism of analytic spaces $\vp: X \to Y$ (over $\KK = \R$ or $\C$)
or, more generally, for flatness over $\OO_Y$ of a coherent sheaf of $\OO_X$-
modules $\FF$.

In the special case that $X=Y$ and $\vp = \id_X$ (the identity morphism
of $X$), our
criterion reduces to the following ``linear algebra criterion''. In a 
neighbourhood of a point $a \in X$, an $\OO_X$-module $\FF$ can be 
presented as
$$
\OO_X^p \stackrel{\Phi}{\longrightarrow} \OO_X^q \longrightarrow \FF
\longrightarrow 0,
$$
where $\Phi$ is given by multiplication by a $q \times p$-matrix of 
analytic functions. Let $r = \rank \Phi(a)$. Then $\FF_a$ is $\OO_{X,a}$-flat
if and only if all minors of order $r+1$ of $\Phi$ vanish near $a$.

Our flatness criterion, in general, is a combination of a condition ``in
codimension zero'' similar to the preceding and a condition ``in codimension
one'' which can be used together with the Weierstrass preparation theorem 
to inductively reduce the fibre-dimension of the morphism $\vp$.

To justify the criterion, we use it to give natural constructive proofs of
several classical results --- Hironaka's existence of the local flattener
\cite{H}, Douady's openness of flatness \cite{Dou}, and Frisch's
generic flatness theorem \cite{Fri}. The proofs are essentially a mix of
linear algebra and appropriate applications of the Weierstrass preparation
theorem.

For example, in the case $X=Y$, the linear algebra criterion above
provides an immediate construction of the \emph{local flattener} of
$\FF$ at $a$ (i.e., the largest germ of an analytic subspace $T$ of $X$
at $a$ such that $\FF_a$ is $\OO_T$-flat). We can simply take $\OO_T =
\OO_X / \II$, where the ideal $\II$ is generated by the minors of order
$r+1$ of $\Phi$. Hironaka's local flattener, in general, can be described
using a similar linear algebra construction and the Weierstrass preparation
theorem.
\medskip

\emph{Algebraic formulation of the flatness criterion.} Let $\vp: Z \to W$
and $\lambda: T \to W$ denote morphisms of analytic space-germs, and let
$F$ denote a finite $\OO_Z$-module. We are concerned with
$\OO_T$-flatness of the module 
$F \tilde{\otimes}_{{\OO}_W} {\OO}_T$, where $\tilde{\otimes}_{{\OO}_W}$ 
denotes the 
analytic tensor product (i.e., the tensor product in the category of local 
analytic 
${\OO}_W$-algebras; see, for example, \cite{ABM1}).
%$\ F {\otimes}_{{\OO}_Y} {\OO}_T\ $.
Via the embedding $\ (\phi, \id_Z) : Z \to W \times Z$ and the natural 
projection $\pi : W \times Z \to W$, 
we can view $F$ as an $\OO_{W \times Z} $-module and therefore as an 
$\OO_W $-module.
Via an embedding 
$Z \hookrightarrow \KK^m_0$ we can also replace $Z$ by $\KK^m_0$
without changing 
the $\OO_W$-module structure of $F$. In particular, then 
$\OO_Z = \KK\{x\} = \KK\{x_1,\dots,x_m\}$, $\OO_W = R/J$ for an appropriate 
ideal $J$ 
in $R := \KK\{y\}=\KK\{y_1,\dots,y_n\}$, and $\OO_{W \times Z} = R\{x\}
:= \KK\{y, x_1,\dots,x_m\}$. 
Let $A := \OO_{W \times Z}$. 
%Let $R=\KK\{y\}=\KK\{y_1,\dots,y_n\}$ be the ring of convergent power series in $n$ variables 
%over the field $\KK$ ($\KK=\R$ or $\C$). Let $J$ be an ideal in $R$, and let 
Let $\mm$ denote the maximal ideal $(y_1,\dots,y_n)$ of $R$, and
%. Let $A=R\{x\}=R\{x_1,\dots,x_m\}$, and let 
let $\nn = \mm+(x_1,\dots,x_m) \subset A$. Then $\nn$ is the maximal ideal
of $A$. Given a power series 
$f=f(y,x)\in A$, we denote by $f(0)$ or by $f(0,x)$ its \emph{evaluation at}
$y=0$; i.e., the 
image of $f$ under the homomorphism $A\to A(0):=A\tensR R/\mm$ of $R$-modules.
%(where $\ \tensR\ $ denotes the analytic tensor product, i.~e. the tensor product in the category of 
%local analytic $\ R$-algebras; see, e.g., \cite{ABM1}). 
Similarly, given an $A$-submodule $M$ of $A^q$, we denote by $M(0)$ the 
\emph{evaluation} of $M$ \emph{at} $y=0$; i.e., $M(0)=\{m(0)\in A(0)^q: 
m\in M\}\ $. 
In particular $A(0)\cong\KK\{x\}$.
\par

We are thus interested in flatness of $F\tensR R/J$ over $R/J$, where 
$F$ is a
finitely generated $A$-module and $J$ is an ideal in $R$.
%as an $R/J$-module.

\begin{theorem}
\label{thm:main} 
Let $R, A, F$ and $J$ be as above.
\smallskip

\noindent
\emph{(A)} There exist $g\in A,\, l\in\N$ and a homomorphism $\psi: A^l\to F$ 
of $A$-modules such that $g(0,x)\neq0,\, g\cdot F\subset\im\psi$ and 
$\ker\psi\subset\mm\cdot A^l$. 
\smallskip

\noindent
\emph{(B)} $F\tensR R/J$ is a flat $R/J$-module if and only if, for any
$g,\,l$ and $\psi$ as in \emph{(A)}, the 
following two conditions hold:
\begin{enumerate}
\item $\ker\psi \subset J\cdot A^l$;
\item $(F/\im\psi)\tensR R/J$ is a flat $R/J$-module.
\end{enumerate}
\end{theorem}

\begin{remark}
\label{rem:fbd-reduction}
The above theorem allows one to study flatness of a module $F$ by repeated
reduction of the 
fibre-dimension over $R$. Indeed, consider $g$ and $\psi$ as in (A). First
suppose that $g(0,0)=0$. 
Since $g(0,x)\neq0$, we can apply the Weierstrass division theorem (after a 
generic linear change in
$x$) to conclude that $A/g\cdot A$ is a finite 
$R\{\tx\}$-module, where $\tx=(x_1,\dots,x_{m-1})$. Then $F/\im\psi$ is a 
finite 
$R\{\tx\}$-module too, since $g\cdot F\subset\im\psi$. On the other hand, if 
$g(0,0)\neq0$ (which 
is the case when the number of $x$-variables is $0$), then condition (2) of 
(B) in the theorem is 
vacuous and no fibre dimension reduction is needed.
\end{remark}

\begin{proof}[Proof of Theorem \ref{thm:main} (A)]
Consider a presentation of $F$ as an $A$-module
\begin{equation}
\label{eqn:pres_F}
A^p\stackrel{\Phi}{\longrightarrow}A^q\stackrel{\Psi}{\longrightarrow}F\to0\,.
\end{equation}
By applying $\tensR R/J$ and $\tensR R/\mm$ to (\ref{eqn:pres_F}), 
we get presentations
\begin{equation}
\label{eqn:pres_FJ}
A^p/J\!\cdot\!A^p\stackrel{\Phi_J}{\longrightarrow}A^q/J\!\cdot\!A^q\stackrel{\Psi_J}{\longrightarrow}F\tensR R/J\to0
\end{equation}
and
\begin{equation}
\label{eqn:pres_Fm}
A^p/\mm\!\cdot\!A^p\stackrel{\Phi_{\mm}}{\longrightarrow}A^q/\mm\!\cdot\!A^q\stackrel{\Psi_{\mm}}{\longrightarrow}F\tensR R/\mm\to0
\end{equation}
of $F\tensR R/J$ and $F\tensR R/\mm$ respectively.
Notice that, identifying $\Phi$ with a matrix (with entries in $A$), 
$\Phi_{\mm}$ becomes 
the matrix with entries obtained by evaluating the corresponding entries 
of $\Phi$ at $y=0$.

Let $r_{\mm}:=\rank(\Phi_{\mm})$. Choose an ordering of the columns and 
rows of $\Phi$ so that $\Phi$ can be written in block form as
\begin{equation}\label{eqn:block}
\Phi=\begin{bmatrix}\alpha & \beta\\ \gamma & \delta\end{bmatrix},
\end{equation}
where the matrix $\alpha$ is of size $r_{\mm}\times r_{\mm}$ and 
$(\det\alpha)(0)=(\det\alpha)(0,x)\neq0$
in $\ A(0)$. 

Let $\alpha^\#$ denote 
the adjoint matrix of $\alpha$; i.e., an $r_{\mm}\times r_{\mm}$ matrix with 
$\alpha^\#\cdot\alpha=\alpha\cdot\alpha^\#=(\det\alpha)\cdot\mathrm{Id}_{r_{\mm}}$.

Now, take $g:=\det\alpha$, $l:=q-r_{\mm}$, and let $\psi$ be the restriction of
$\Psi:A^{r_{\mm}}\oplus A^l\to F\ $ to $\ \{0\}^{r_{\mm}}\oplus A^l \cong A^l$.
Then $g(0,x)\neq0$. The condition $g\cdot F\subset\im\psi$ is equivalent to 
saying that, for every vector $(\vr,\sigma)\in A^{r_{\mm}}\oplus A^l$, 
there exists 
$\sigma'\in A^l$ such that $\Psi(g\cdot(\vr,\sigma))=\Psi((0,\sigma'))$, or,
equivalently, that 
$g\cdot A^q\subset\ker\Psi+(\{0\}^{r_\mm}\oplus A^l)=\im\Phi+(\{0\}^{r_\mm}\oplus A^l)$. But the 
latter follows from the fact that $g\cdot A^{r_{\mm}}\subset\im\alpha$.

Finally, by the choice of $\psi$, $\sigma\in\ker\psi$ if and only if 
$(0,\sigma)\in\im\Phi\cap(\{0\}^{r_\mm}\oplus A^l)$. Then $(0,\sigma)=\Phi((\xi,\eta))$, for 
some $(\xi,\eta)\in A^{r_{\mm}}\oplus A^{p-r_{\mm}}$ with 
$\alpha \xi + \beta \eta=0$. By 
the choice of $r_{\mm}$, every row of $[\gamma, \delta]$ is an $A(0)$-linear 
combination 
of the rows of $[\alpha, \beta]$ modulo $\mm$. Hence 
$\alpha \xi + \beta \eta =0$ implies that 
$\gamma \xi+\delta \eta \in\mm\cdot A^l$; i.e., 
that $\sigma\in \mm\cdot A^l$.
\end{proof}

Theorem \ref{thm:main}(B) is the main result of this article. We will
prove it in Section~\ref{sec:proof-main} below.

\begin{remark}
\label{rem:identity}
Throughout the paper, we will use the fact that all entries of the matrix 
$g\cdot\delta-\gamma\cdot\alpha^\#\cdot\beta$ are 
$(r_{\mm}+1)\times(r_{\mm}+1)$ minors of $\Phi$. This is an immediate 
consequence of the following matrix identity: For any $q \times p$ block
matrix (\ref{eqn:block}), where $\alpha$ is of size $r \times r$, 
\begin{equation}
\label{enq:identity}
g\cdot\Phi=\begin{bmatrix}\alpha & 0\\ \gamma & \mathrm{Id}_{q-r}\end{bmatrix}\cdot \begin{bmatrix}g\cdot\mathrm{Id}_r & \alpha^\#\cdot\beta\\ 0 & g\cdot\delta-\gamma\cdot\alpha^\#\cdot\beta\end{bmatrix}\,,
\end{equation}
where $g=\det\alpha$.
\end{remark}

%%%%%%%%%%%%%%%%%%%%%%%%%%%%%%%%%%%%%%%%%%%%%%%%%%
%%%%%%%%%%%%%%%%%%%%%%%%%%%%%%%%%%%%%%%%%%%%%%%%%%
%%%%%%%%%%%%%%%%%%%%%%%%%%%%%%%%%%%%%%%%%%%%%%%%%%
\section{Applications: local flattener, openness of flatness, generic flatness.}
\label{sec:applications}

\begin{theorem}[{Hironaka's local flattener~\cite{H}}]
\label{thm:Hironaka}
Let $\vp:Z\to W$ be a morphism of analytic space-germs, where $W$ is regular.
Let $F$ be a finite $\OO_Z$-module.
Then there exists a unique analytic subgerm $P$ of $W$ (i.e., a unique local 
analytic $\KK$-algebra $\OO_P$, which is a quotient of $\OO_W$) such that:
\begin{itemize}
\item[(1)] $F \antens_{\OO_W} \OO_P$ is $\OO_P$-flat.
\item[(2)] Let $\lb_P:P\to W$ denote the embedding. Then, for every morphism 
$\lb: T \to W$ of germs of analytic spaces such that $F \antens_{\OO_W}\OO_T$ 
is $\OO_T$-flat, there exists a unique morphism $\mu:T\to P$ such that $\lb=\lb_P\circ\mu$.
\end{itemize}
\end{theorem}

\begin{remark}
\label{rem:embed}
Suppose that $\lb:T\to W$ is a morphism such that $F \antens_{\OO_W}\OO_T$ is 
$\OO_T$-flat. Since flatness is preserved by base change (see 
\cite[Prop.\,6.8]{H}), it follows that 
$(F \antens_{\OO_W}\OO_T)\antens_{\OO_T}S$ is $S$-flat, for every subring $S$ 
of $\OO_T$. In particular, identifying $\OO_W/\ker\lb^*$ with $\im\lb^*$, 
we get that $F \antens_{\OO_W}(\OO_W/\ker\lb^*)\,\cong\,
(F \antens_{\OO_W}\OO_T)\antens_{\OO_T}(\OO_W/\ker\lb^*)$ is 
$(\OO_W/\ker\lb^*)$-flat. Therefore, in Theorem~\ref{thm:Hironaka} it suffices 
to consider an embedding $\lb:T \to W$, and to show that there is an ideal 
$I(F)$ in $\OO_W$ such that $F \antens_{\OO_W}(\OO_W/J)$ is $\OO_W/J$-flat 
if and only if $I(F)\subset J$.

The germ $P$ is called the \emph{local flattener} of $F$ (with respect to 
$\vp$), and $I(F)$ is the \emph{ideal of the local flattener}.
\end{remark}

\begin{proof}[Proof of Theorem~\ref{thm:Hironaka}]
The uniqueness of $P$ is automatic, since $\lb_P^*:\OO_W\to\OO_P$ is surjective.

By regularity of $W$, we can identify $\OO_W$ with the ring $R=\KK\{y\}$ of 
convergent power series in  $y=(y_1,\dots,y_n)$.
Assume that $Z$ is a subgerm of $\KK^m_0$. Using the graph of $\vp$ to embed $Z$ 
in $W\times\KK^m$, we can think of $\OO_Z$ as a quotient ring of $A=R\{x\}$, 
where $x=(x_1,\dots,x_m)$. Then $F$ is a finitely generated $A$-module. We 
will proceed by induction on $m$, the number of the $x$-variables.

Choose $g\in A$ and $\psi:A^l\to F$ satisfying Theorem~\ref{thm:main}(A). 
Let $J(F)$ be the ideal in $R$ generated by the coefficients of (the expansions
in $x$ of) the elements in $\ker\psi$; i.e., the unique minimal ideal $J$ in 
$R$ satisfying $\ker\psi\subset J\!\cdot\!A^l$. If $F=\im\psi$ (which is the 
case if $m=0$, since then $g$ is invertible in $A$), then 
Theorem~\ref{thm:main}(B) implies that $J(F)$ is the ideal of the local 
flattener of $F$. If $F\neq\im\psi$, then $m>0$ and we may assume by the 
inductive hypothesis (see Remark~\ref{rem:fbd-reduction}) that there is a
local flattening 
ideal $I(F/\im\psi)$ in $\OO_W$. It follows that $I(F):=J(F)+I(F/\im\psi)$ 
is the ideal of the local flattener of $F$.
\end{proof}

Let $X$ and $Y$ be analytic spaces over $\KK$, and let $\vp:X\times Y\to Y$ 
be the canonical projection.
Let $\FF$ be a coherent $\OO_{Y\times X}$-module. For $(\eta,\xi)\in Y\times X$,
let $I_{\eta,\xi}(\FF)$ denote the ideal in $\OYeta$ of the local flattener of 
the stalk $\FF_{(\eta,\xi)}$ (with respect to $\vp$). Given any ideal $J$ in 
$\OYeta$, we let $J_{\eta'}$ denote the ideal generated by (a system of 
generators of) $J$ at nearby points $\eta' \in Y$. Then 
Theorem~\ref{thm:main} implies the following.

\begin{theorem}[Openness of flatness]
\label{thm:open-flat}
For every $(\eta,\xi)$ in a sufficiently small open neighbourhood of $(\eta_0,\xi_0)$ in $Y\times X$, with $\eta$ in a representative of the zero-set germ $\VV(I_{\eta_0,\xi_0}(\FF))$, we have
\[
I_{\eta,\xi}(\FF)\,\subset\,(I_{\eta_0,\xi_0}(\FF))_\eta\,.
\]
\end{theorem}

\begin{remark}[{Douady's openness of flatness~\cite{Dou}}]
\label{rem:Douady}
Let $\vp:X\to Y$ be a morphism of analytic spaces, and let $\FF$ be a coherent sheaf of $\OO_X$-modules.
Let $J$ be a coherent sheaf of ideals in $\OO_Y$, and let $Z$ be the closed 
analytic subspace of $Y$ defined by $J$ (i.e., $\OO_Z=\OO_Y/J$ and $|Z|=\supp(\OO_Y/J)$).
Then Theorem~\ref{thm:open-flat} implies that
\[
N_X(Z)=\{\xi\in\vp^{-1}(|Z|):\,\FF_\xi\antens_{\OO_{Y,\vp(\xi)}}\OO_{Z,\vp(\xi)}\mathrm{\ is\ not\ }\OO_{Z,\vp(\xi)}\mathrm{-flat}\,\}
\]
is a closed subset of $|X|$. In particular, for $Z=Y$, the latter implies 
openness of the set of points $\xi\in X$ with the property that $\FF_\xi$ is a 
flat $\OO_{Y,\vp(\eta)}$-module. This result is due to Douady~\cite{Dou} and 
is the classical form of ``openness of flatness''.
\end{remark}

\begin{proof}[Proof of Theorem~\ref{thm:open-flat}]
As in the proof of Theorem~\ref{thm:Hironaka}, we proceed by induction on the
fibre-dimension $m$ of $\vp:X\times Y\to Y$.
Using Theorem \ref{thm:main}(A) with $F=\FF_{(\eta_0,\xi_0)}$, we can choose
neighbourhoods $U$ of $\xi_0$ and $V$ of $\eta_0$, a function $g$ analytic on 
$V\times U$, and a morphism $\psi:\OO^l_{V\times U}\to\FF|_{V\times U}$ of 
$\OO_{V\times U}$-modules, such that $g(\eta_0,x)\neq0$, 
$g_{(\eta_0,\xi_0)}\!\cdot\FF_{(\eta_0,\xi_0)}\subset(\im\psi)_{(\eta_0,\xi_0)}$
and $(\ker\psi)_{(\eta_0,\xi_0)}\subset\mm_{V,\eta_0}\cdot\OO^l_{V\times U,(\eta_0,\xi_0)}$.
Since our problem is local, we can assume that $U$ (resp. $V$) is an open 
polydisc in $\C^m$ (resp. $\C^n$) centred at $\xi_0$ (resp. $\eta_0$).
(After shrinking $V$ if necessary), let $J$ be a coherent $\OO_V$-ideal 
such that $J_{\eta_0} = I_{\eta_0,\xi_0}(\FF)$; we can assume that 
$J_\eta=(I_{\eta_0,\xi_0}(\FF))_\eta$ for all $\eta\in V$. Let $Z$ denote the 
closed analytic subspace of $V$ defined by $J$; i.e., $|Z|$ is a 
representative in $V$ of the zero-set germ $\VV(I_{\eta_0,\xi_0}(\FF))$. 
Then Theorem~\ref{thm:main}(B) implies that
\begin{gather}
\label{eq:cond1}
(\ker\psi)_{(\eta_0,\xi_0)}\subset\,J_{\eta_0}\cdot\,\OO^l_{V\times U,(\eta_0,\xi_0)}\,,\\
\label{eq:cond2}
(\FF/\im\psi)_{(\eta_0,\xi_0)}\antens_{\OO_{Y,\eta_0}}\OO_{Z,\eta_0}\ \,\mathrm{is\ }\, \OO_{Z,\eta_0}\mathrm{-flat}\,.
\end{gather}
It follows (after shrinking $U$ and $V$ if needed) that  
$g(\eta,x)\neq0$ for all $\eta\in V$ and $g\!\cdot\!\FF\subset\im\psi$. Then
\eqref{eq:cond1} implies
\begin{equation}
\label{eq:cond3}
(\ker\psi)_{(\eta,\xi)}\subset\,J_\eta\!\cdot\OO^l_{V\times U,(\eta,\xi)}\subset\,\mm_{V,\eta}\!\cdot\OO^l_{V\times U,(\eta,\xi)}\,,
\end{equation}
for all $(\eta,\xi)\in V\times U$ with $\eta\in|Z|$. 

If $g(\eta,\xi)\neq0$ (which is the case if $m=0$), then,
 by Theorem~\ref{thm:main}(B), the first inclusion of \eqref{eq:cond3} 
implies that $I_{\eta,\xi}(\FF)\subset J_\eta=(I_{\eta_0,\xi_0}(\FF))_\eta$, 
as required.

Otherwise $g(\eta_0,\xi_0)=0$ (and $m>0$). By Theorem~\ref{thm:main}, it 
suffices to show that 
$(\FF/\im\psi)_{(\eta,\xi)}\antens_{\OO_{Y,\eta}}\OO_{Z,\eta}$ is 
$\OO_{Z,\eta}$-flat, provided $\eta\in|Z|$ and $g(\eta,\xi)=0$.
After a linear change of the $x$-variables, we can assume that $U=U'\times U''$,
where $U'$ is spanned by the variables $\tx=(x_1,\dots,x_{m-1})$ and $U''$ is 
spanned by $x_m$, and $g_{(\eta_0,\xi_0)}$ is regular in $x_m-\xi_{0m}$, 
where $\xi_{0m}$ is the last coordinate of $\xi_0$.
By Remark~\ref{rem:fbd-reduction}, after shrinking $U$ if needed, we can 
consider $\FF/\im\psi$ as a coherent $\OO_{V\times U'}$-module; we denote it
$\tilde\FF$. Let $\txi_0$ denote the $\tx$-coordinates of $\xi_0$. 
Then $\tilde\FF_{(\eta_0,\txi_0)}=(\FF/\im\psi)_{(\eta_0,\xi_0)}$ (since 
$g(\eta_0,\txi_0,\cdot)$ vanishes only at $\xi_{0m}$), and hence 
$\tilde\FF_{(\eta_0,\txi_0)}\antens_{\OO_{Y,\eta_0}}\OO_{Z,\eta_0}$
is $\OO_{Z,\eta_0}$-flat, by \eqref{eq:cond2}. By the inductive hypothesis, 
$\tilde\FF_{(\eta,\txi)}\antens_{\OO_{Y,\eta}}\OO_{Z,\eta}$ is 
$\OO_{Z,\eta}$-flat for every $(\eta,\txi)\in|Z|\times U'$. To complete the 
proof, observe that for any $(\eta,\xi)\in|Z|\times U$ with $g(\eta,\xi)=0$, 
$(\FF/\im\psi)_{(\eta,\xi)}$ is a direct summand of $\tilde\FF_{(\eta,\txi)}$. 
Hence 
$(\FF/\im\psi)_{(\eta,\xi)}\antens_{\OO_{Y,\eta}}\OO_{Z,\eta}$ is 
$\OO_{Z,\eta}$-flat, as a direct summand of 
$\tilde\FF_{(\eta,\txi)}\antens_{\OO_{Y,\eta}}\OO_{Z,\eta}$, 
by \cite[Ch.\,1, \S\,2.3, Prop.\,2]{Bou}.
\end{proof}

\begin{remark}[{Frisch's generic flatness theorem \cite{Fri}}]
\label{rem:Frisch}
Let $\vp: X \to Y$ denote a morphism of complex-analytic spaces and
let $\FF$ denote a coherent sheaf of $\OO_X$-modules. Frisch's
\emph{generic flatness theorem} asserts that the \emph{non-flat locus}
$\Sigma := \{\xi \in X: \FF_\xi \text{ is not } \OO_{Y,\vp(\xi)}\text{-flat}\}$
is a closed analytic subset of $X$, and, if $X$ is reduced, then $\vp(\Sigma)$
is nowhere dense in $Y$. Frisch's theorem follows from Theorem 
\ref{thm:open-flat} above together with the fact that $\Sigma$ is a
constructible subset of $X$. See \cite[Thm.\,7.15]{BMPisa} for a constructive
elementary proof of the latter.
\end{remark}

%%%%%%%%%%%%%%%%%%%%%%%%%%%%%%%%%%%%%%%%%%%%%%%%%%
%%%%%%%%%%%%%%%%%%%%%%%%%%%%%%%%%%%%%%%%%%%%%%%%%%
%%%%%%%%%%%%%%%%%%%%%%%%%%%%%%%%%%%%%%%%%%%%%%%%%%
\section{Proof of the main theorem}
\label{sec:proof-main}

We use the notation preceding Theorem \ref{thm:main}. 
Consider a presentation \eqref{eqn:pres_F} of $F$ as an $A$-module.
Applying $\tensR R/\mm$, we get  a homomorphism $\Phi_{\mm}:A(0)^p\to A(0)^q$ 
of $A(0)$-modules such that $F\tensR R/\mm\cong\mathrm{coker}(\Phi_{\mm})$. 
Set $r_{\mm}:=\rank(\Phi_{\mm})$. We can assume that $\Phi$ is given by
a block matrix \eqref{eqn:block} and $g:=\det\alpha$ 
satisfies $g(0,x)\neq0$. For an ideal $J$ in $R$, define
$$
\ker_J\Phi:=\{\zeta\in A^p:\Phi(\zeta)\in J\!\cdot\!A^q\}\,,
$$
and
$$
\rank_J\Phi:=\min\{r\geq1:\ \mathrm{all}\ (r+1)\times(r+1)\mathrm{\ minors\ of\ }\Phi\mathrm{\ belong\ to\ }J\!\cdot\!A\}\,.
$$

Our proof of Theorem~\ref{thm:main}(B) is based on showing that property (1) of the theorem is equivalent to equalities $q-l=\rank_J\Phi=\rank\Phi_{\mm}$, and that property (2) of the theorem is equivalent to $R/J$-flatness of $\GG\tensR R/J$, where
\[
\GG:=A^{r_{\mm}}/[g\cdot A^{r_{\mm}}+\im(\alpha^\#\cdot\beta)]\,.
\]
The latter equivalence is obvious if $g$ is a unit in $A$, since both $F/\im\psi$ and $\GG$ are zero in this case.
Suppose then that $g$ is not invertible in $A$, that is, $g(0,0)=0$. Since $g(0,x)\neq0$, then after a (generic and linear) change of the $x$-coordinates to $(\tx,x_m)$, where $\tx=(x_1,\dots,x_{m-1})$, we have $g(0,0,x_m)\neq0$. By the Weierstrass Preparation Theorem, $g=u\cdot P$, where $u(0,0)\neq0$ and $P(y,x)=x_m^d+\sum_{i=1}^dp_i(y,\tx)\cdot x_m^{d-i}$, with $p_i(0,0)=0$.

The ring $A/g\!\cdot\!A$ is a finite free $R\{\tx\}$-module.
We shall describe the action of $\alpha^\#\cdot\beta:A^{p-r_{\mm}}\to A^{r_{\mm}}$ modulo $g$ as linear mapping of finite $R\{\tx\}$-modules. Given $\eta\in A^{p-r_{\mm}}$, Weierstrass division by $g$ gives $\eta\equiv\sum_{j=1}^d\eta_jx_m^{d-j}$ (mod $g$), with $\eta_j\in R\{\tx\}^{p-r_{\mm}}$. Applying Weierstrass division by $g$ to the entries of $\alpha^\#\cdot\beta$, we form matrices $T_i=T_i(y,\tx)$, $1\leq i\leq d$, such that
\begin{equation}
\label{eq:mod1g}
(\alpha^\#\cdot\beta)(\eta)\equiv(\sum_{i=1}^dT_i\cdot x_m^{d-i})\cdot(\sum_{j=1}^d\eta_j\cdot x_m^{d-j})\quad (\mathrm{mod\;}g)\,.
\end{equation}
Applying Euclid division by $P(y,x)$ (as a monic polynomial in $x_m$) to the latter product, we obtain matrix $G=(G_{ij})_{1\leq i,j\leq d}$, with block-matrices $G_{ij}$ of size $r_{\mm}\times(p-r_{\mm})$ and entries in $R\{\tx\}$, such that all entries of the matrix
\begin{equation}
\label{eq:mod2g}
(\sum_{i=1}^dT_i\cdot x_m^{d-i})\cdot(\sum_{j=1}^d\eta_j\cdot x_m^{d-j})\ -\ \sum_{1\leq i,j\leq d}G_{ij}\cdot\eta_j\cdot x_m^{d-i}
\end{equation}
are linear in the $\eta_j$ with coefficients in the ideal generated by $P(y,x)$ in the ring $R\{\tx\}[x_m]$. Then $\GG$ coincides with $R\{\tx\}^{r_{\mm}d}/\im\!G$ as $R\{\tx\}$-modules. With these preparations and modulo Lemma~\ref{lem:kernels} below, Theorem~\ref{thm:main}(B) is a consequence of the following:

\begin{proposition}
\label{prop:main}
Let $G:R\{\tx\}^{(p-r_{\mm})d}\to R\{\tx\}^{r_{\mm}d}$ be as above (or $G=0$ if $g(0,0)\neq0$). Then $\ker(\Phi_{\mm})=(\ker_J\Phi)(0)$ if and only if $\rank\Phi_{\mm}=\rank_J\Phi$ and $\ker(G_{\mm})=(\ker_J G)(0)$.
\end{proposition}

Before proving Proposition \ref{prop:main}, let us note that
the first equality of Proposition~\ref{prop:main} expresses $R/J$-flatness of 
$F$:

\begin{lemma}
\label{lem:kernels}
$F\tensR R/J$ is $R/J$-flat if and only if $(\ker_J\Phi)(0)=\ker(\Phi_{\mm})$.
\end{lemma}

\begin{proof}
By definition of $\ker_J\Phi$, $\zeta\in\ker_J\Phi$ implies $\,\Phi(\zeta)\in\mm\!\cdot\!A^q$, and hence $\,\Phi_{\mm}(\zeta(0))=0$. Therefore, we always have $(\ker_J\Phi)(0)\subset\ker\Phi_{\mm}$. On the other hand, by a well-known criterion for flatness (see, e.g., \cite[Prop.\,6.2]{H}), $F\tensR R/J$ is $R/J$-flat iff $\anTor^{R/J}_1(F\tensR R/J, R/\mm)=0$.

By (\ref{eqn:pres_FJ}), we have $F\tensR R/J\cong(A^q/J\!\cdot\!A^q)/\Phi_J(A^p/J\!\cdot\!A^p)$. Notice that $\ker(\Phi_J)=(\ker_J\Phi)/J\!\cdot\!A^p$. Hence $\Phi_J(A^p/J\!\cdot\!A^p)\cong (A^p/J\!\cdot\!A^p)/\ker(\Phi_J)\cong A^p/\ker_J\Phi$, and we get from (\ref{eqn:pres_FJ}) a short exact sequence
\[
0\,\to\,A^p/\ker_J\Phi\,\to\,A^q/J\!\cdot\!A^q\,\to\,F\tensR R/J\,\to\,0\,.
\]
The induced long exact sequence of the $\anTor^{R/J}$ modules ends with
\begin{multline}
\notag
0\,\to\,\anTor^{R/J}_1(F\tensR R/J,R/\mm)\,\to\,(A^p/\ker_J\Phi)\antens_{R/J}R/\mm\\
\stackrel{\lb}{\rightarrow}\,(A^q\tensR R/J)\antens_{R/J}R/\mm\,\to\,(F\tensR R/J)\antens_{R/J}R/\mm\,\to\,0\,,
\end{multline}
where the leftmost term is zero by $R/J$-flatness of $A^q\tensR R/J$ (which follows from the $R$-flatness of $A^q$).
Thus $F\tensR R/J$ is $R/J$-flat if and only if
\[
A(0)^p/(\ker_J\Phi)(0)\cong(A^p/\ker_J\Phi)\antens_{R/J}R/\mm\,\stackrel{\lb}{\longrightarrow}\,A^q\antens_{R/J}R/\mm\cong A(0)^q
\]
is injective. By (\ref{eqn:pres_Fm}), the latter condition is equivalent to $(\ker_J\Phi)(0)\supset\ker(\Phi_{\mm})$, which completes the proof of the lemma.
\end{proof}
\medskip

The proof of Proposition~\ref{prop:main} depends on several lemmas following. 
First, we establish a useful 
cancellation law.

\begin{lemma}
\label{lem:cancel}
Let $J$ be an ideal in $R$, and let $g,\zeta\in A$ be such that $g(0,x)\neq0$ in $A(0)=\KK\{x\}$ and $g\cdot\zeta\in J\!\cdot\!A$. Then $\zeta\in J\!\cdot\!A$.
\end{lemma}

\begin{proof}
Write $\zeta=\sum_{\nu\in\N^m}\zeta_{\nu}x^{\nu}$, where $\zeta_{\nu}\in R$, and consider $g$ and $\zeta$ as elements of the ring $\tilde{A}:=R[[x]]$.
By assumption, $g\notin\mm\!\cdot\!\tilde{A}$. Hence, after localizing in $\mm\!\cdot\!\tilde{A}$, we get $\zeta_{\mm\tilde{A}}\in(J\tilde{A})_{\mm\tilde{A}}$, because $g_{\mm\tilde{A}}$ is invertible in $\tilde{A}_{\mm\tilde{A}}$.
Since $\tilde{A}$ is a free $R$-module, we have $\tilde{A}_{\mm\tilde{A}}\cong R_{\mm}[[x]]$, and hence $\zeta_{\mm\tilde{A}}\in(J\tilde{A})_{\mm\tilde{A}}$ if and only if, for all $\nu\in\N^m$, $(\zeta_{\nu})_{\mm}\in J_{\mm}$, that is, $\zeta_{\nu}\in J$.
Thus $\zeta\in J\!\cdot\!A$, as required.
\end{proof}

Recall that $r_{\mm}$ denotes the rank of $\Phi_{\mm}$ (in the notation at
the beginning of this section).

\begin{lemma}
\label{lem:6equiv}
Let $J$ be an ideal in $R$. Then the following conditions are equivalent:
\begin{itemize}
\item[(i)] $\rank\Phi_{\mm}=\rank_J\Phi$;
\smallskip

\item[(ii)] we can order the columns and rows of $\Phi$ so that
$\Phi$ has block form 
\eqref{eqn:block} with $\alpha$ of size $r\times r$,
$(\det\alpha)(0,x)\neq0$ and $\rank_J\Phi=r$;
\smallskip

\item[(iii)] we can order the columns and rows of $\Phi$
so that $\Phi$ has block form
\eqref{eqn:block}, where $\alpha$ has size 
$r\times r$, $(\det\alpha)(0,x)\neq0$, and all entries of 
$(\det\alpha)\cdot\delta-\gamma\cdot\alpha^\#\cdot\beta$ are in $J\!\cdot\!A$;
\smallskip

\item[(iv)] if $\Phi$ is a block matrix \eqref{eqn:block}, where $\alpha$ is 
of size $r_{\mm}\times r_{\mm}$ and $(\det\alpha)(0)\neq0$, then all entries 
of $(\det\alpha)\cdot\delta-\gamma\cdot\alpha^\#\cdot\beta$ are in 
$J\!\cdot\!A$;
\smallskip

\item[(v)] if $g\in A$, $g(0,x)\neq0$, and $A^q=A^r\oplus A^l$, where 
$g\cdot A^q\subset\im\Phi+(\{0\}^r\oplus A^l)$ and 
$\im\Phi\cap(\{0\}^r\oplus A^l)\subset\{0\}^r\oplus\mm\!\cdot\!A^l$, then 
$\im\Phi\cap(\{0\}^r\oplus A^l)\subset\{0\}^r\oplus J\!\cdot\!A^l$;
\smallskip

\item[(vi)] if $g\in A$, $g(0,x)\neq0$ and $\psi:A^l\to F$ is a homomorphism of $A$-modules such that $g\cdot F\subset\im\psi$ and $\ker\psi\subset\mm\!\cdot\!A^l$, then $\ker\psi\subset J\!\cdot\!A^l$.
\end{itemize}
\end{lemma}

\begin{proof}
$\mathrm{(ii)}\Rightarrow\mathrm{(i)}$: Clearly $r\leq\rank\Phi_{\mm}$ and $\rank\Phi_{\mm}\leq\rank_J\Phi$. Hence all three are equal if $\rank_J\Phi=r$.
\smallskip

$\mathrm{(i)}\Rightarrow\mathrm{(iv)}$: By Remark~\ref{rem:identity}, all entries of $(\det\alpha)\cdot\delta-\gamma\cdot\alpha^\#\cdot\beta$ are $(r_{\mm}+1)\times(r_{\mm}+1)$ minors of $\Phi$, and hence they belong to $J\!\cdot\!A$ if $\rank_J\Phi=r_{\mm}$.
\smallskip

$\mathrm{(iv)}\Rightarrow\mathrm{(iii)}$: Set $r=r_{\mm}$ and let $\alpha,\beta,\gamma,\delta$ be as in (iv).
\smallskip

$\mathrm{(iii)}\Rightarrow\mathrm{(ii)}$: Set $g=\det\alpha$. By the matrix 
identity of Remark~\ref{rem:identity}, all $(r+1)\times(r+1)$ minors of 
$g\cdot\Phi$ are combinations of the entries of 
$(\det\alpha)\cdot\delta-\gamma\cdot\alpha^\#\cdot\beta$ with coefficients in 
$A$. Hence, if $\zeta$ is an $(r+1)\times(r+1)$ minor of $\Phi$, then 
$g^{r+1}\cdot\zeta\in J\!\cdot\!A$, which by Lemma~\ref{lem:cancel} 
implies $\zeta\in J\!\cdot\!A$.
\smallskip

$\mathrm{(v)}\Rightarrow\mathrm{(vi)}$: The homomorphism $\psi:A^l\to F$ can be extended to a surjective homomorphism $\Psi:A^q\to F$, which by Oka's coherence theorem extends to an exact sequence $A^p\stackrel{\Phi}{\longrightarrow}A^q\stackrel{\Psi}{\longrightarrow}F\to0$.
\smallskip

$\mathrm{(vi)}\Rightarrow\mathrm{(v)}$: The assumptions in (v) imply the assumptions in (vi), with the same $g$ and $\psi$ being the restriction of $\Psi$ (from the above exact sequence) to $\{0\}^r\oplus A^l$. Then $\ker\psi=\im\Phi\cap(\{0\}^r\oplus A^l)\subset J\!\cdot\!A^l$.
\smallskip

It remains to show that (iv) is equivalent to (v). Write $\Phi$ in block form
\eqref{eqn:block}, with $\alpha$ of size $r\times r$. We will use the fact 
that $(\vr,\sigma)\in A^q=A^r\oplus A^l$ belongs to 
$\im\Phi\cap(\{0\}^r\oplus A^l)$ if and only if 
$\sigma=\gamma\xi+\delta\eta$ and $\alpha\xi+\beta\eta=\vr=0$, for 
some $(\xi,\eta)\in A^r\oplus A^{p-r}$. Then 
$(\det\alpha)\cdot\xi=(\alpha^\#\cdot\alpha)(\xi)=-(\alpha^\#\cdot\beta)(\eta)$,
and hence
\[
(\det\alpha)\cdot\sigma=\gamma((\det\alpha)\cdot\xi)+(\det\alpha)\cdot\delta(\eta)=-(\gamma\cdot\alpha^\#\cdot\beta)(\eta)+(\det\alpha)\cdot\delta(\eta)\,.
\]
It follows that
\begin{equation}
\label{eqn:lem}
(\det\alpha)\cdot[\im\Phi\cap(\{0\}^r\oplus A^l)]\ \subset\ \{0\}^r\oplus\im[(\det\alpha)\cdot\delta-\gamma\cdot\alpha^\#\cdot\beta]
\subset\ \im\Phi\cap(\{0\}^r\oplus A^l)\,,
\end{equation}
where the latter inclusion is a consequence of Remark~\ref{rem:identity}.
\smallskip

$\mathrm{(v)}\Rightarrow\mathrm{(iv)}$: The assumptions of (iv) imply that all entries of $(\det\alpha)\cdot\delta-\gamma\cdot\alpha^\#\cdot\beta$ are in $\mm\!\cdot\!A$ (by Remark~\ref{rem:identity}, as $(r_{\mm}+1)\times(r_{\mm}+1)$ minors of $\Phi$). Therefore the assumptions of (v) follow with $r:=r_{\mm}$, $l:=q-r$ and $g:=\det\alpha$. Indeed, $g\cdot\mathrm{Id}_r=\alpha\cdot\alpha^\#$ and so
\[
g\cdot A^q\ \subset\ \alpha(A^r)\oplus A^l\ \subset\ \im\Phi+(\{0\}^r\oplus A^l)\,.
\]
Also, by (\ref{eqn:lem}), $\zeta=(\vr,\sigma)\in\im\Phi\cap(\{0\}^r\oplus A^l)$ implies $g\cdot\sigma\in\im[(\det\alpha)\cdot\delta-\gamma\cdot\alpha^\#\cdot\beta]$. Hence $g\cdot\zeta\in\mm\!\cdot\!A^q$, and therefore $\zeta\in\mm\!\cdot\!A^q$, by Lemma~\ref{lem:cancel}.

Now, (v) implies $\im\Phi\cap(\{0\}^r\oplus A^l)\subset\{0\}^r\oplus J\!\cdot\!A^l$, which by (\ref{eqn:lem}) means that
$\im[(\det\alpha)\cdot\delta-\gamma\cdot\alpha^\#\cdot\beta]\subset J\!\cdot\!A^l$, and hence the entries of $(\det\alpha)\cdot\delta-\gamma\cdot\alpha^\#\cdot\beta$ are in $J\!\cdot\!A$.
\smallskip

$\mathrm{(iv)}\Rightarrow\mathrm{(v)}$: Let $\pi_1:A^q=A^r\oplus A^l\to A^r$ 
denote the canonical projection to the first direct summand. By the 
assumptions of (v), there is a matrix $\Xi$ of size $p\times r$ with entries 
in $A$, such that $g\cdot\mathrm{Id}_r=\pi_1\cdot\Phi\cdot\Xi$. Since 
$g(0,x)\neq0$, it follows that $\rank(\pi_1\cdot\Phi)=r$. Therefore there is 
an ordering of columns of $\Phi$ such that $\pi_1\cdot\Phi=[\alpha,\beta]$, 
with $\alpha$ of size $r\times r$ and $(\det\alpha)(0,x)\neq0$. Then $\Phi$
has block form \eqref{eqn:block} and 
$\{0\}^r\oplus\im[(\det\alpha)\cdot\delta-\gamma\cdot\alpha^\#\cdot\beta]\subset\im\Phi\cap(\{0\}^r\oplus A^l)$; hence, by the assumptions of (v), all entries 
of $(\det\alpha)\cdot\delta-\gamma\cdot\alpha^\#\cdot\beta$ are in 
$\mm\cdot A^l$. Using the equivalence of (ii) and (iii) for $J=\mm$, we see 
that $r=\rank_\mm\Phi=\rank\Phi_\mm$; i.e., the assumptions of (iv) are 
satisfied. It follows that 
$J\cdot A^l\supset \im[(\det\alpha)\cdot\delta-\gamma\cdot\alpha^\#\cdot\beta]$,
hence $\{0\}^r\oplus J\cdot A^l\supset (\det\alpha)\cdot[\im\Phi\cap(\{0\}^r\oplus A^l)]$, by~\eqref{eqn:lem}, 
and thus $\{0\}^r\oplus J\cdot A^l\supset \im\Phi\cap(\{0\}^r\oplus A^l)$, 
by Lemma~\ref{lem:cancel}.
\end{proof}

\begin{lemma}
\label{lem:ker2rank}
Assume that $\ker(\Phi_{\mm})=(\ker_J\Phi)(0)$. Then $\rank\Phi_{\mm}=\rank_J\Phi$.
\end{lemma}

\begin{proof}
Clearly, $r_\mm=\rank\Phi_\mm\leq\rank_J\Phi$.
For the opposite inequality, choose $\xi_j(x)\in\ker\Phi_\mm\subset\KK\{x\}^p$, $1\leq j\leq p-r_\mm$, so that the $p\times(p-r_\mm)$ matrix $\xi(x)=[\xi_1(x),\dots,\xi_{p-r_\mm}(x)]$ has rank $p-r_\mm$. Then, by assumption, there is a matrix $\Xi=\Xi(y,x)$ of size $p\times(p-r_\mm)$ such that the entries of $\Phi\cdot\Xi$ are in $J\cdot A$ and $\Xi(0,x)=\xi(x)$. It follows that $\rank\Xi=p-r_\mm$. By Cramer's Rule (and after an appropriate reordering of the columns of $\Phi$ and rows of $\Xi$), there exists a matrix $\Sigma$ of size $(p-r_\mm)\times(p-r_\mm)$ with entries in $A$ such that
\[
\Xi\cdot\Sigma\,=\,\begin{bmatrix}g\cdot\mathrm{Id}_{p-r_\mm}\\ \Gamma\end{bmatrix}\,,
\]
where $g\in A$ satisfies $g(0,x)\neq0$, and $\Gamma$ is a matrix with entries in $A$ of size $r_\mm\times(p-r_\mm)$. Write $\Phi=[\Phi_1,\Phi_2]$, where $\Phi_1$ consists of the first $p-r_\mm$ columns of $\Phi$. It follows that $g\cdot\Phi_1+\Phi_2\cdot\Gamma$ is a matrix with entries in $J\cdot A$, and hence the entries of $\;g\cdot\Phi-[-\Phi_2\cdot\Gamma,g\cdot\Phi_2]\;$ are in $J\cdot A$, too. Since $\Phi_2$ is of size $q\times r_\mm$, then $\rank[-\Phi_2\cdot\Gamma,g\cdot\Phi_2]\leq\rank\Phi_2\leq r_\mm$. Consequently,
\[
\rank_J(g\cdot\Phi)=\rank_J[-\Phi_2\cdot\Gamma,g\cdot\Phi_2]\leq r_\mm\,.
\]
It thus suffices to show that $\rank_J\Phi=\rank_J(g\cdot\Phi)$, but that is a consequence of Lemma~\ref{lem:cancel}.
\end{proof}

\begin{remark}
\label{rem:proj-ker}
Let $\Phi$ be as in Lemma~\ref{lem:6equiv}\,(iv), and let $\pi_2:A^p=A^{r_\mm}\oplus A^{p-r_\mm}\to A^{p-r_\mm}$ denote the canonical projection to the second direct summand. Then
\[
(\ker_J\Phi)(0)=\ker(\Phi_\mm)\ \ \mathrm{iff}\ \ \pi_2((\ker_J\Phi)(0))=\pi_2(\ker(\Phi_\mm))\,,
\]
where $J$ is an ideal in $R$.
Indeed, since $(\ker_J\Phi)(0)$ is always contained in $\ker(\Phi_\mm)$ (cf. proof of Lemma~\ref{lem:kernels}), it suffices to show that  $\;\pi_2((\ker_J\Phi)(0))\supset\pi_2(\ker(\Phi_\mm))$ implies $\ker(\Phi_\mm)\subset(\ker_J\Phi)(0)$. Let $\zeta=\zeta(x)$ be an element of $\ker\Phi_\mm$, and let $\xi\in\ker_J\Phi$ be such that $\pi_2(\xi(0,x))=\pi_2(\zeta)$. It suffices to show that $\zeta(x)=\xi(0,x)$. Since $\eta(x):=\xi(0,x)-\zeta(x)$ belongs to $\ker\pi_2\cap\ker\Phi_\mm$, it follows that $\eta=(\eta',0)\in A^{r_\mm}\oplus A^{p-r_\mm}$ and $\alpha(0,x).\eta'(x)=0$. Therefore $(\det\alpha)(0,x)\cdot\eta'(x)=0$, hence $\eta'=0$, and $\eta=0$, as required.
\end{remark}

\begin{lemma}
\label{lem:mod-J}
Let $\Phi$ and $\pi_2:A^p=A^{r_\mm}\oplus A^{p-r_\mm}\to A^{p-r_\mm}$ be as above, and let $J$ be an ideal in $R$. Then $\eta\in\pi_2(\ker_J\Phi)$ if and only if the following two conditions hold
\begin{gather*}
(\alpha^\#\cdot\beta)(\eta)\in g\cdot A^{r_\mm}+J\cdot A^{r_\mm}\\
(g\cdot\delta-\gamma\cdot\alpha^\#\cdot\beta)(\eta)\in J\cdot A^{q-r_\mm}\,,
\end{gather*}
where $g$ denotes $\det\alpha$.
\end{lemma}

\begin{proof}
For the ``only if'' direction, let $(\xi,\eta)$ be an element of $\ker_J\Phi$. Then $\alpha\xi+\beta\eta\in J\cdot A^{r_\mm}$ and $\gamma\xi+\delta\eta\in J\cdot A^{q-r_\mm}$, hence
\begin{gather*}
g\cdot\xi+(\alpha^\#\cdot\beta)(\eta)=\alpha^\#\cdot(\alpha\xi+\beta\eta)\equiv0\ (\mathrm{mod\ }J\!\cdot\! A^{r_\mm}),\quad\mathrm{and}\\
(g\!\cdot\!\delta-\gamma\!\cdot\!\alpha^\#\!\cdot\!\beta)(\eta)=g\!\cdot\!(\gamma\xi+\delta\eta)-\gamma\!\cdot\!(g\!\cdot\!\xi+(\alpha^\#\!\cdot\!\beta)(\eta))\equiv0\ (\mathrm{mod\ }J\!\cdot\! A^{q-r_\mm}).
\end{gather*}
Now, for the ``if'' direction, let $\xi\in A^{r_\mm}$ be such that $g\cdot\xi+(\alpha^\#\cdot\beta)(\eta)\equiv0$ modulo $J\!\cdot\!A^{r_\mm}$, and assume that $(g\cdot\delta-\gamma\cdot\alpha^\#\cdot\beta)(\eta)\in J\cdot A^{q-r_\mm}$. Then
\begin{gather*}
g\cdot(\alpha\xi+\beta\eta)=\alpha\cdot(g\cdot\xi+(\alpha^\#\cdot\beta)(\eta))\equiv0\ (\mathrm{mod\ }J\!\cdot\!A^{r_\mm}),\quad\mathrm{and}\\
g\!\cdot\!(\gamma\xi+\delta\eta)=(g\!\cdot\!\delta-\gamma\!\cdot\!\alpha^\#\!\cdot\!\beta)(\eta)+\gamma\!\cdot\!(g\!\cdot\!\xi+(\alpha^\#\!\cdot\!\beta)(\eta))\equiv0\ (\mathrm{mod\ }J\!\cdot\!A^{q-r_\mm}).
\end{gather*}
Therefore $g\cdot(\xi,\eta)\in\ker_J\Phi$, hence $(\xi,\eta)\in\ker_J\Phi$ by Lemma~\ref{lem:cancel}, and so $\eta\in\pi_2(\ker_J\Phi)$, as required.
\end{proof}

\begin{remark}
\label{rem:eta-in-pi}
Since the entries of $g\!\cdot\!\delta-\gamma\!\cdot\!\alpha^\#\!\cdot\!\beta$ are in $\mm\!\cdot\!A$ (by Remark~\ref{rem:identity}), Lemma~\ref{lem:mod-J} applied to $J=\mm$ asserts that
\[
\eta\in\pi_2(\ker(\Phi_\mm))\quad\mathrm{iff}\quad(\alpha^\#\cdot\beta)(0,x).\eta(x)\in g(0,x)\cdot A(0)^{r_\mm}\,.
\]
\end{remark}

\begin{proof}[Proof of Proposition~\ref{prop:main}]

Let $\pi_2:A^p=A^{r_\mm}\oplus A^{p-r_\mm}\to A^{p-r_\mm}$ be as above. By Lemma~\ref{lem:ker2rank} and Remark~\ref{rem:proj-ker}, it suffices to show the equivalence
\[
\pi_2((\ker_J\Phi)(0))=\pi_2(\ker(\Phi_\mm))\quad\mathrm{iff}\quad(\ker_J G)(0)=\ker(G_{\mm})\,,
\]
under the assumption that $r_\mm=\rank_J\Phi$ (i.e., the equivalent conditions of Lemma \ref{lem:6equiv} are satisfied).

Suppose first that $g$ is a unit in $A$ (and hence $(\ker_JG)(0)=\ker(G_\mm)$ trivially).
Then the condition $\;(\alpha^\#\cdot\beta)(\eta)\in g\cdot A^{r_\mm}+J\cdot A^{r_\mm}\;$ of Lemma~\ref{lem:mod-J} is vacuous, because $g\!\cdot\!A^{r_\mm}+J\!\cdot\!A^{r_\mm}=A^{r_\mm}$. Since the entries of $g\!\cdot\!\delta-\gamma\!\cdot\!\alpha^\#\!\cdot\!\beta$ are in $J\!\cdot\!A$ (Lemma~\ref{lem:6equiv}\,(iv)), it follows from Lemma~\ref{lem:mod-J} that $\pi_2(\ker_J\Phi)=A^{p-r_\mm}$.
Therefore $\;\pi_2((\ker_J\Phi)(0))\supset\pi_2(\ker(\Phi_\mm))$, and hence $(\ker_J\Phi)(0)=\ker(\Phi_\mm)$, by Remark~\ref{rem:proj-ker}.
\smallskip

Suppose then that $g$ is not a unit in $A$; i.e., $g(0,0)=0$.
Let $\eta=\eta(x)\in A(0)^{p-r_\mm}$. Since $g(0,x)\neq0$, then after a generic linear change of the $x$-variables, $g$ is regular in $x_m$. Applying the Weierstrass division theorem, we get
\[
\eta(x)=\sum_{j=1}^d\eta_j(\tilde x)\cdot x_m^{d-j}+g(0,x)\cdot\tilde q(x)\,,
\]
where $\tilde x=(x_1,\dots,x_m)$. Hence, by Remark~\ref{rem:eta-in-pi},
\[
\eta\in\pi_2(\ker(\Phi_\mm))\quad\mathrm{iff}\quad (\alpha^\#\!\cdot\!\beta)(0,x).(\sum_{j=1}^d\eta_j(\tilde x)\cdot x_m^{d-j})\in g(0,x)\cdot A(0)^{r_\mm}\,.
\]
By \eqref{eq:mod1g} and \eqref{eq:mod2g}, the latter is the case iff $\{\eta_j(\tilde x)\}_{j=1}^d\in\ker(G_\mm)$.

Finally, let $\eta=\eta(y,x)\in A^{p-r_\mm}$. By the Weierstrass division theorem (after a linear change of the $x$-variables, if needed),
\[
\eta(y,x)=\sum_{j=1}^d\eta_j(y,\tilde x)\cdot x_m^{d-j}+g(y,x)\cdot\tilde q(y,x)\,,
\]
where $\tilde x=(x_1,\dots,x_m)$. Since the entries of $g\!\cdot\!\delta-\gamma\!\cdot\!\alpha^\#\!\cdot\!\beta$ are in $J\!\cdot\!A$ (Lemma~\ref{lem:6equiv}\,(iv)), Lemma~\ref{lem:mod-J} implies that
\[
\eta\in\pi_2(\ker_J\Phi)\quad\mathrm{iff}\quad(\alpha^\#\!\cdot\!\beta)(\sum_{j=1}^d\eta_j(y,\tilde x)\cdot x_m^{d-j})\in g\cdot A^{r_\mm}+J\cdot A^{r_\mm}\,.
\]
By \eqref{eq:mod1g} and \eqref{eq:mod2g}, the latter is the case iff 
$\{\eta_j(y,\tilde x)\}_{j=1}^d\in\ker_J G$, which completes the proof of 
Proposition~\ref{prop:main} and Theorem~\ref{thm:main}.
\end{proof}

%%%%%%%%%%%%%%%%%%%%%%%%%%%%%%%%%%%%%%%%%%%%%%%%%%
%References
%%%%%%%%%%%%%%%%%%%%%%%%%%%%%%%%%%%%%%%%%%%%%%%%%%
\bibliographystyle{amsplain}

\end{document}